\documentclass[runningheads,a4paper]{llncs}

\usepackage{amssymb}
\usepackage{amsmath}
\usepackage{graphicx}

\usepackage{url}

\newtheorem{thm}{Theorem}
\newtheorem{lem}{Lemma}

\newcommand{\E}{\mathbb{E}}
\newcommand{\Prob}{\mathsf{P}}

\date{}

\begin{document}

\mainmatter

\title{Global clustering coefficient in scale-free networks}

\titlerunning{Global clustering coefficient in scale-free networks}

\author{Liudmila~ Ostroumova Prokhorenkova\inst{1}\inst{2}
\and Egor~Samosvat\inst{1}\inst{3}}

\institute{Yandex, Moscow, Russia
\and
Moscow State University, Moscow, Russia
\and
Moscow Institute of Physics and Technology, Moscow, Russia}

\maketitle

\begin{abstract}
In this paper, we analyze the behavior of the global clustering coefficient in scale free graphs. We are especially interested in the case of degree distribution with an infinite variance, since such degree distribution is usually observed in real-world networks of diverse nature.

There are two common definitions of the clustering coefficient of a graph: global clustering and average local clustering. It is widely believed that in real networks both clustering coefficients tend to some positive constant as the networks grow. There are several models for which the average local clustering coefficient tends to a positive constant. On the other hand, there are no models of scale-free networks with an infinite variance of degree distribution and with a constant global clustering.

In this paper we prove that if the degree distribution obeys the power law with an infinite variance, then the global clustering coefficient tends to zero with high probability as the size of a graph grows.

\end{abstract}

\section{Introduction}\label{sec:intro}

In this paper, we analyze the global clustering coefficient of graphs with a power-law degree distribution.
Namely, we consider a sequence of graphs with the degree distribution following a regularly varying distribution $F$.
Our main result is the following.
If the degree distribution has an infinite variance, then the global clustering coefficient tends to zero with high probability.

It is important to note that we do not specify any random graph model, our result holds for any sequence of graphs.
The only restriction we have on a sequence is the distribution of degrees: we assume that the degrees of vertices (except one vertex, see the explanation at the end of Section~\ref{sec:graph}) are i.i.d.
random variables following a regularly varying distribution with a parameter $1 < \gamma <~2$.

Our results are especially interesting taking into account the fact that
it was suspected that for many types of networks both the average local and the global  clustering coefficients tend to non-zero limit as the network becomes large.
It is a natural assumption as in many observed networks the values of both clustering coefficients are considerably high \cite{Newman}.
Note that actually these observations do not contradict ours:
\begin{itemize}
\item Large values of global clustering coefficient are usually obtained on small networks.
\item For the networks with the power-law degree distribution the observed global clustering is usually less than the average local clustering, as expected. 
\item Our results can be applied only to networks with regularly varying degree distribution. If a network has, for example, a power-law degree distribution with an exponential cut-off, then our results cannot be applied.
\end{itemize}

The rest of the paper is organized as follows.
In the next section, we discuss two definitions of the clustering coefficient.
Then, in Section~\ref{sec:graph}, we formally define our restriction on the sequence of graphs.
In Section~\ref{sec:existence}, we prove that a simple graph with the given degree sequence exists with high probability.
In Section~\ref{sec:global_clustering}, we prove that the global clustering coefficient for any such sequence of graphs tends to zero.
Then we discuss one graph constructing procedure which gives a sequence of graphs with superlinear number of triangles, but the global clustering coefficient for such sequence still tends to zero.
Section~\ref{sec:conclusion} concludes the paper.

\section{Clustering coefficients}\label{sec:clustering}

There are two popular definitions of the clustering coefficient \cite{Math_Results,Newman}. The \emph{global clustering coefficient} $C_1(G_n)$ is the ratio of three times the number of triangles to the number of pairs of adjacent edges in $G_n$. The \emph{average local clustering coefficient} is defined as follows: $C_2(G_n) = \frac{1}{n} \sum_{i=1}^n C(i)$, where $C(i)$ is the local clustering coefficient for a vertex $i$: $C(i) = \frac{T^i}{P_2^i}$, where $T^i$ is the number of edges between the neighbors of the vertex $i$ and $P_2^i$ is the number of pairs of neighbors.
Note that both clustering coefficients equal $1$ for a complete graph.

It was mentioned in \cite{Math_Results,Newman} that in research papers either average local or global clustering are considered.
And it is not always clear which definition is used.
On the other hand, these two clustering coefficients differ.
It was demonstrated in \cite{GPA} that for networks based on the idea of preferential attachment the difference between these two clustering coefficients is crucial.



Note that both definitions of the clustering coefficient work only for graphs without multiple edges.
Also, most measurements on real-world networks do not take multiple edges into account.
Therefore, further we consider only simple graphs: graphs without loops and multiple edges.
Clustering coefficient for weighted graphs can also be defined (see, e.g., \cite{MultClust}).
We leave the analysis of the clustering coefficient in weighted graphs for the future work



\section{Scale-free graphs}\label{sec:graph}

We consider a sequence of graphs $\{G_n\}$.
Each graph $G_n$ has $n$ vertices.
We assume that the degrees of these vertices are independent random variables following a {\it regularly varying} distribution with a cumulative distribution function $F$ such that:
\begin{equation}\label{eq:regular}
1-F(x)=L(x)x^{-\gamma},\quad x>0,
\end{equation}
where $L(\cdot)$ is a slowly varying function, that is, for any fixed constant $t>0$
\[\lim_{x\to\infty}\frac{L(tx)}{L(x)}=1.\]
There are other obvious restrictions on the function $L(\cdot)$,
for instance, the function $1-L(x)x^{-\gamma}$ must be a cumulative distribution function of a random variable taking positive integer values with probability 1.
Further in this paper we use the following property of slowly varying functions:
$L(x) = o\left(x^c\right)$ for any $c>0$.

Note that \eqref{eq:regular} describes a broad class of heavy-tailed distributions without imposing the rigid Pareto assumption.
Power-law distribution with parameter $\gamma + 1$ corresponds to the cumulative distribution $1-F(x)=L(x) x^{-\gamma}$.
Further by $\xi, \xi_1, \xi_2, \ldots$ we denote random variables with the distribution $F$.
Note that for any $\alpha < \gamma$ the moment $\E \xi^{\alpha}$ is finite.



Models with $\gamma > 2$ and with the global clustering coefficient tending to some positive constant were already proposed (see, e.g., \cite{GPA}).
Therefore, in this paper we consider only the case $1<\gamma < 2$.

One small problem remains: we can construct a graph with a given degree distribution only if the sum of degrees is even. This problem is easy to solve: we can either regenerate
the degrees until their sum is even or we can add 1 to the last variable if their sum is odd \cite{Configuration}.
For simplicity we choose the second option, i.e., if $\sum_{i=1}^n \xi_i$ is odd, then we replace $\xi_n$ by $\xi_n+1$.
It is easy to see that this correction does not change any of our results, therefore,
further we do not focus on the evenness.

\section{Existence of a graph with given degree distribution}\label{sec:existence}

\subsection{Result}

As pointed out in \cite{Molloy}, the probability of obtaining a simple graph with given degree distribution by random pairing of edges' endpoints (configuration model) converges to a strictly positive constant if the degree distribution has a finite second moment.
In our case the second moment is infinite and it can be shown that the probability of obtaining a simple graph just by random pairing of edges' endpoints tends to zero with $n$.

However, we can prove that in this case a simple graph with a given degree distribution exists with high probability
and it can be constructed, e.g., using Havel-Hakimi algorithm \cite{Hakimi,Havel}.

\begin{thm}\label{thm:existence}
For any $\delta$ such that $1< \delta < \gamma$ with probability $1-O\left(n^{1-\delta}\right)$
there exists a simple graph on $n$ vertices with the degree distribution defined above.

\end{thm}

\subsection{Auxiliary results}

We use the following theorem proved by Erd\H{o}s and Gallai in 1960~\cite{ErdosGallai}.

\begin{thm}[Erd\H{o}s--Gallai]
A sequence of non-negative integers ${d_1 \geq \ldots \geq d_n}$ can be represented as the degree sequence of a finite simple graph on $n$ vertices if and only if
\begin{enumerate}
\item $d_1+\ldots+d_n$ is even;
\item $\sum^{k}_{i=1}d_i\leq k(k-1)+ \sum^n_{i=k+1} \min(d_i,k)$ holds for $1\leq k\leq n$.
\end{enumerate}
\end{thm}

In this case a sequence ${d_1 \geq \ldots \geq d_n}$ is called \emph{graphic}.

If a degree sequence is graphic, then one can use Havel-Hakimi algorithm to construct a simple graph corresponding to  it \cite{Hakimi,Havel}.
The idea of the algorithm is the following. We sort degrees in nondecreasing order $d_1 \ge \ldots \ge d_n$.
Then we take the vertex of the highest degree $d_1$ and connect this vertex to the vertices of degrees $d_2, \ldots, d_{d_1+1}$.
After this we get the degree sequence $d_2-1, \ldots, d_{d_1+1}-1, d_{d_1+2}, \ldots, d_n$ and apply the same procedure to this sequence, and so on.

We also use the following theorem several times in this paper (see, e.g., \cite{RegVar}).

\begin{thm}[Karamata's theorem]\label{thm:karamata}
Let $L$ be slowly varying and locally bounded in $[ x_0,\infty ]$ for some $x_0 \geq 0$. Then

\begin{enumerate}
\item for $\alpha > -1$
    $$
        \int_{x_0}^x t^{\alpha} L(t) dt  = (1+o(1))  (\alpha +1)^{-1} x^{\alpha+1} L(x),\,\,\,\, x\to\infty\,.
    $$
\item for $\alpha < -1$
    $$
        \int_{x}^{\infty} t^{\alpha} L(t) dt   = - (1+o(1)) (\alpha +1)^{-1} x^{\alpha+1} L(x),\,\,\,\, x\to\infty\,.
    $$
\end{enumerate}
\end{thm}

We also use the following known lemma (proof can be found, e.g., in~\cite{fresh_model}).

\begin{lem}\label{lem:moments}
Let $\xi_1, \dots, \xi_n$ be mutually independent random variables, $\E \xi_i = 0$, $\E |\xi_i|^{\alpha} < \infty$, $1\le \alpha \le 2$, then
$$
\E \left[ |\xi_1 + \ldots + \xi_n|^{\alpha} \right] \le 2^{\alpha} \left(\E\left[|\xi_1|^{\alpha}\right]+\ldots+\E\left[|\xi_n|^{\alpha}\right] \right)\,.
$$
\end{lem}

\subsection{Proof of Theorem~\ref{thm:existence}}

We need the following lemma on the number of edges in the graph.

\begin{lem}\label{lem:edges}
For any $\theta$ such that $1<\theta<\gamma$ with probability $1-O(n^{1-\theta})$ the number of edges $E(G_n)$ in our graph satisfies the following inequalities:
$$
\frac{n \E \xi}{4} \le E(G_n) \le \frac{3n\E\xi}{4}\,.
$$
\end{lem}

\begin{proof}

The expectation of the number of edges is
$$
\E E(G_n) = n \E \xi / 2 \,.
$$

Note that for  $1 < \theta < \gamma$ we have $\E|\xi - \E\xi|^\theta < \infty$ and
\begin{multline*}
\Prob \left(|E(G_n) - n\E\xi/2| \ge n\E\xi/4\right) \le
\frac{4^\theta \E \left| \sum_{i=1}^n \left( \xi_i - \E\xi\right)/2\right|^\theta}{n^{\theta} (\E\xi)^\theta}
\\\le \frac{ 8^\theta n \E|\xi - \E\xi|^\theta}{n^{\theta} (\E\xi)^\theta}
= O\left(n^{1-\theta}\right)\,.
\end{multline*}
Here we applied Lemma~\ref{lem:moments}.
This concludes the proof of Lemma~\ref{lem:edges}.

\end{proof}

Let us order the random variables $\xi_1, \ldots, \xi_n$ and obtain the ordered sequence $d_1 \ge \ldots \ge d_n$.

We want to show that with probability $1-O\left(n^{1-\delta}\right)$ the condition

\begin{equation}\label{eq:degcondition}
\sum^{k}_{i=1}d_i\leq k(k-1)+ \sum^n_{i=k+1} \min(d_i,k)
\end{equation}
holds for all $k$, $1\leq k\leq n$.

Note that if $k \geq \sqrt{2\E\xi n}$, then with probability $1-O\left( n^{1-\delta}\right)$ we have
$$
\sum^{k}_{i=1}d_i \leq 2 E(G_n) \leq k(k-1)
$$
as  $ 2 E(G_n) \le \frac{3\E\xi}{2}n $ (here we apply Lemma~\ref{lem:edges} with $\theta=\delta$).
Therefore the condition~\eqref{eq:degcondition} is satisfied.

Now consider the case $k < \sqrt{2\E\xi n}$.
In this case we will show that
$$
\sum^n_{i=k+1} \min(d_i,k) \ge \sum^{k}_{i=1}d_i\,
$$
which implies the condition \eqref{eq:degcondition}.
Note that $\min(d_i,k) > 1$ so
$$
\sum^n_{i=k+1} \min(d_i,k)   \ge n - \sqrt{2\E\xi n}\,.
$$
It remains to show that with probability $1-O\left(n^{1-\delta}\right)$ we have
\begin{equation}\label{eq:degcondition2}
\sum^{\sqrt{2\E\xi n}}_{i=1}d_i \le n - \sqrt{2\E\xi n}\,.
\end{equation}

Fix some $\alpha$ such that $\delta < \alpha < \gamma$.
Consider any $\beta$ such that
\begin{equation}\label{eq:beta}
0 < \beta < \min\left\{\frac{2-\delta}{\gamma},\frac{1}{2\gamma},\frac{\alpha - \delta}{\gamma(\alpha-1)}\right\}
\end{equation}
and let
$$
S_n = \sum_{i=1}^n \xi_i I\left[\xi_i>n^{\beta}\right]\,.
$$
We will show that with probability $1-O\left(n^{1-\delta}\right)$ we have
\begin{equation}\label{eq:degcondition3}
\sum^{\sqrt{2\E\xi n}}_{i=1}d_i \le S_n \le n - \sqrt{2\E\xi n}
\end{equation}
which implies \eqref{eq:degcondition2}.
Note that in order to prove the left inequality it is sufficient to show that with probability $1-O\left(n^{1-\delta}\right)$ we have
$$
S'_n := \sum_{i=1}^n I\left[\xi_i>n^{\beta}\right] \ge \sqrt{2\E\xi n}\,.
$$

The expectation of $S'_n$ is
$$
\E S'_n = \E\sum_{i=1}^n I\left[\xi_i>n^{\beta}\right] =
n\Prob\left(\xi>n^{\beta}\right) =
nL\left(n^{\beta}\right) n^{-\gamma\beta}\,.
$$
Now we will show the concentration:
$$
\Prob \left( |S'_n - \E S'_n| > \frac{\E S'_n}{2} \right) \le
\frac{4 \mathrm{Var}(S'_n)}{(\E S'_n)^2} =
$$
$$
= \frac{4n \left(L\left(n^{\beta}\right) n^{-\gamma\beta} - \left(L\left(n^{\beta}\right)\right)^2 n^{-2\gamma\beta} \right)}{n^2 \left(L\left(n^{\beta}\right)\right)^2 n^{-2\gamma\beta}} =
O \left( \frac{n^{\gamma\beta}  }{n L\left(n^{\beta}\right)} \right) = O \left( n^{1-\delta} \right)\,.
$$

Here in the last equation we use the inequality $\beta < \frac{2-\delta}{\gamma}$, so $\gamma\beta-1 < 1 - \delta$.
It remains to note that as $\beta < \frac{1}{2\gamma}$ for large enough $n$ we have
$$
\frac 1 2 nL\left(n^{\beta}\right) n^{-\gamma\beta} \ge \sqrt{2 \E \xi n}\,.
$$

Now let us prove the right inequality in \eqref{eq:degcondition3}, i.e., prove that with probability $1 - O\left(n^{1-\delta}\right)$ we have
$$
S_n \le n - \sqrt{2\E\xi n}\,.
$$

As before, first we estimate the expectation of $S_n$:
\begin{multline*}
\E S_n = n \int_{n^{\beta}}^{\infty} x d F(x) =
-n \int_{n^{\beta}}^{\infty} x \, d (1- F(x))
\\
= - n \, x (1- F(x)) \bigg|_{n^{\beta}}^{\infty} + n \int_{n^{\beta}}^{\infty} (1-F(x)) \, d x
\\
= n \, n^{\beta} n^{-\gamma\beta} L\left(n^{\beta}\right) + n \int_{n^{\beta}}^{\infty} x^{-\gamma}L(x) \, d x
\\
\sim n^{1+\beta(1-\gamma)} L\left(n^{\beta}\right) + n (\gamma-1)^{-1} n^{\beta(1-\gamma)}  L\left(n^{\beta}\right) = \frac{\gamma}{\gamma-1}n^{1+\beta(1-\gamma)} L\left(n^{\beta}\right)\,.
\end{multline*}
In order to show concentration we first estimate
\begin{multline*}
\E \left(\xi I\left[\xi >n^{\beta}\right]\right)^\alpha =
- \int_{n^{\beta}}^{\infty} x^\alpha \, d (1- F(x))
\\
= - x^\alpha (1- F(x)) \bigg|_{n^{\beta}}^{\infty} + \int_{n^{\beta}}^{\infty} (1-F(x)) \, d x^\alpha
\\
= n^{\alpha\beta} n^{-\gamma\beta} L\left(n^{\beta}\right) + \alpha \int_{n^{\beta}}^{\infty} x^{\alpha-\gamma-1}L(x) \, d x
\\
\sim n^{\beta(\alpha-\gamma)} L\left(n^{\beta}\right) + (\gamma-\alpha)^{-1} n^{\beta(\alpha-\gamma)}  L\left(n^{\beta}\right) = \frac{\gamma+1-\alpha}{\gamma-1}n^{\beta(\alpha-\gamma)} L\left(n^{\beta}\right)\,.
\end{multline*}
We get
\begin{multline*}
\Prob \left( |S_n - \E S_n| > \frac{\E S_n}{2} \right) \le
\frac{\E |S_n - \E S_n|^\alpha}{(\E S_n)^\alpha}
\\
= O\left( \frac{ n \E \left(\xi I\left[\xi >n^{\beta}\right]\right)^\alpha}{(\E S_n)^\alpha} \right)
= O \left( \frac{n^{1+\beta(\alpha-\gamma)} L\left(n^{\beta}\right)}{n^{\alpha(1+\beta(1-\gamma))} \left(L\left(n^{\beta}\right)\right)^\alpha} \right) = O\left(n^{1-\delta}\right)\,.
\end{multline*}
Here in the last equation we use the inequality $\beta < \frac{\alpha - \delta}{\gamma(\alpha-1)}$.

It remains to note that as $\gamma > 1$ for large $n$ we have
$$
\frac{\gamma}{2(\gamma-1)}n^{1+\beta(1-\gamma)} L\left(n^{\beta}\right) < n - \sqrt{2\E\xi n}\,.
$$

\section{Global clustering coefficient}\label{sec:global_clustering}

\subsection{Result}

\begin{thm}\label{thm:cluster}
For any $\varepsilon>0$ and any $\alpha$ such that $1<\alpha<\min\left\{2 - \frac{\gamma}{2}, \gamma \right\}$ with probability $1-O(n^{1-\alpha})$ the global clustering coefficient satisfies the following inequality
$$
C_1(G_n) \le n^{\varepsilon - \frac{\left(\gamma - 2\right)^2}{2\gamma}}\,.
$$
\end{thm}
Taking small enough $\varepsilon$ one can see that with high probability $C_1(G_n) \to 0$ as $n$ grows.

\subsection{Proof of Theorem~\ref{thm:cluster}}

We will use the following estimate for $C_1(G_n)$:
\begin{equation}\label{eq:cluster}
C_1(G_n) \le \frac{E(G_n) \left|\{i: \xi_i^2 \ge E(G_n)\}\right|  + \sum_{i: \xi_i^2<E(G_n)} \xi_i^2}{P_2(n)}.
\end{equation}
Here $P_2(n)$ is the number of pairs of adjacent edges in $G_n$.
In order to obtain inequality \eqref{eq:cluster} we use the following observation. The number of triangles connected to a vertex cannot be larger than the number of edges in a graph. Also, this number cannot be larger than the degree squared.

Using Lemma~\ref{lem:edges} (with $\theta=\alpha$) we get that with probability $1 - O\left(n^{1-\alpha}\right)$
\begin{equation}\label{eq:clust}
C_1(G_n) \le \frac{ \frac{3\E\xi n}{4} \left|\left\{i: \xi_i^2 \ge \frac{\E\xi n}{4}\right\}\right|  + \sum_{i: \xi_i^2<\frac{3\E\xi n}{4}} \xi_i^2 }{P_2(n)}.
\end{equation}

Let us find a lower bound for the number of pairs of adjacent edges $P_2(n)$.
\begin{lem}\label{lem:p2}
For any $\delta > 0$  and any $\alpha$ such that $1<\alpha<\gamma$  with probability $1-O\left(n^{1-\alpha}\right)$ we have
$$
P_2(n) \ge n^{\frac{2}{\gamma}- \delta}\,.
$$
\end{lem}

\begin{proof}

Let $\xi_{max} = \max\{\xi_1, \ldots, \xi_n\}$, then
$$
P_2(n) \ge  \frac{\xi_{max}(\xi_{max}-1)}{2}\,.
$$

It remains to find a lower bound for $\xi_{max}$ now:
\begin{multline*}
\Prob(\xi_{max} < 2 n^{\frac{1}{\gamma}-\frac{\delta}{2}})
= \left[ \Prob\left(\xi < 2 n^{\frac{1}{\gamma}-\frac{\delta}{2}} \right) \right]^n
= \exp\left(n \log \left(1 - \Prob(\xi \ge 2 n^{\frac{1}{\gamma}-\frac{\delta}{2}})\right)\right) \\
= \exp\left(n \log \left(1 - L\left(2 n^{\frac{1}{\gamma}-\frac{\delta}{2}}\right) 2^{-\gamma}n^{-\gamma\left(\frac{1}{\gamma}-\frac{\delta}{2}\right)}\right)\right)
\\
= \exp\left(- n \left( L\left(2 n^{\frac{1}{\gamma}-\frac{\delta}{2}}\right) 2^{-\gamma}n^{-1+\gamma\frac{\delta}{2}}\right) (1 + o(1))\right)
\\
= \exp\left(-L\left(2 n^{\frac{1}{\gamma}-\frac{\delta}{2}}\right) 2^{-\gamma} n^{\gamma\frac{\delta}{2}} (1 + o(1))\right) = O\left(n^{1-\alpha}\right).
\end{multline*}

So, with probability $1-O\left(n^{1-\alpha}\right)$ we have
$$
P_2(n) \ge n^{\frac{2}{\gamma}-\delta}\,.
$$
This concludes the proof of Lemma~\ref{lem:p2}

\end{proof}

Now we estimate the number of vertices with large degrees.
\begin{lem}\label{lem:large}
For any $\delta > 0$  and any $\alpha$ such that $1<\alpha<\gamma$  we have
$$
\Prob\left(\left|\left\{i: \xi_i \ge \sqrt{\frac{\E\xi n}{4}}\right\}\right| \le n^{1-\frac{\gamma}{2}+\delta}\right) = 1 - O\left(n^{1-\alpha}\right)\,.
$$
\end{lem}

\begin{proof}

Let

$$
S'_n := \sum_{i=1}^n I\left[\xi_i \geq \sqrt{\frac{\E\xi n}{4}}\right] \,.
$$

The expectation of $S'_n$ is
\begin{multline*}
\E S'_n = \E\sum_{i=1}^n I\left[\xi_i\geq\sqrt{\frac{\E\xi n}{4}}\right] =
n\Prob\left(\xi\geq\sqrt{\frac{\E\xi n}{4}}\right)
\\
= n\left({\frac{\E\xi n}{4}}\right)^{-\gamma/2}L\left(\sqrt{\frac{\E\xi n}{4}}\right) \,.
\end{multline*}
We can apply Chernoff bound:
$$
\Prob \left( S'_n  > 2 \, \E S'_n \right) \le
\exp \left(\E S'_n / 3 \right)
= O \left(n^{1-\alpha} \right) \,.
$$


It remains to note that for large enough $n$ we have
$$
2 \E S_n' < n^{1-\frac{\gamma}{2}+\delta}\,.
$$

\end{proof}

\begin{lem}\label{lem:small}
For any $\delta > 0$  and any $\alpha$ such that $1<\alpha<2 - \frac \gamma 2$  we have
$$
\Prob\left(\sum_{i: \xi_i^2<\frac{3 \E\xi n}{4}} \xi_i^2 \le  n^{2-\gamma/2 + \delta}  \right) =  1 - O\left(n^{1-\alpha}\right)\,.
$$
\end{lem}

\begin{proof}
Let
$$
S_n = \sum_{i=1}^n \xi_i^2 I\left[ \xi_i< \sqrt{\frac{3 \E\xi n}{4}} \right]\,.
$$
Again, we first estimate the expectation of $S_n$. Since $L(x)$ is locally bounded we can apply Karamata's theorem:
\begin{multline*}
\E S_n = - n \int_{0}^{\sqrt{\frac{3 \E\xi n}{4}}} x^2 d (1-F(x))
\\
= - n \, x^2 (1- F(x)) \bigg|_{0}^{\sqrt{\frac{3 \E\xi n}{4}}} + 2 n \int_{0}^{\sqrt{\frac{3 \E\xi n}{4}}} x (1-F(x)) \, d x
\\
= - n \, \left(\frac{3 \E\xi n}{4}\right)^{1-\gamma/2} L\left(\sqrt{\frac{3 \E\xi n}{4}}\right) + 2 n \int_{0}^{\sqrt{\frac{3 \E\xi n}{4}}} x^{1-\gamma}L(x) \, d x
\\
\sim - n \, \left(\frac{3 \E\xi n}{4}\right)^{1-\gamma/2} L\left(\sqrt{n}\right) + 2 n (2-\gamma)^{-1} \left(\frac{3 \E\xi n}{4}\right)^{1-\gamma/2}  L\left(\sqrt{n}\right)
\\
=\frac{\gamma}{2-\gamma} \left(\frac{3 \E\xi}{4}\right)^{1-\gamma/2} n^{2-\gamma/2} L\left(\sqrt n \right)\,.
\end{multline*}
In order to show concentration we first estimate
\begin{multline*}
\E \left(\xi^2 I\left[ \xi< \sqrt{\frac{3 \E\xi n}{4}} \right]\right)^2 =
- \int_{0}^{\sqrt{\frac{3 \E\xi n}{4}}} x^{4} \, d (1- F(x))
\\
= - x^{4} (1- F(x)) \bigg|_{0}^{\sqrt{\frac{3 \E\xi n}{4}}} + \int_{0}^{\sqrt{\frac{3 \E\xi n}{4}}}  (1-F(x)) \, d x^{4}
\\
=  - \left(\frac{3 \E\xi n}{4}\right)^{2-\gamma/2} L\left(\sqrt{\frac{3 \E\xi n}{4}}\right) + 4 \int_{0}^{\sqrt{\frac{3 \E\xi n}{4}}} x^{4-\gamma-1}L(x) \, d x
\\
\sim - \left(\frac{3 \E\xi n}{4}\right)^{2-\gamma/2} L\left(\sqrt{n}\right) +
 (4-\gamma)^{-1} \left(\frac{3 \E\xi n}{4}\right)^{2-\gamma/2}  L\left(\sqrt{n}\right)
\\
= O\left(n^{2-\gamma/2} L\left(\sqrt{n}\right)\right)\,.
\end{multline*}
And we get
\begin{multline*}
\Prob \left( |S_n - \E S_n| > \frac{\E S_n}{2} \right) \le
\frac{4 \mathrm{Var} (S_n) }{(\E S_n)^2} 
\le \frac{4 n \E \left(\xi^2 I\left[ \xi< \sqrt{\frac{3 \E\xi n}{4}} \right]\right)^2}{(\E S_n)^2} \\
= O \left(  \frac{n^{3-\gamma/2} L\left(\sqrt{n}\right)}{n^{4-\gamma} \left(L\left(\sqrt{n}\right)\right)^2} \right) 
= O \left(  \frac{n^{\gamma/2 - 1} }{ L\left(\sqrt{n}\right) } \right)
= O\left(n^{1 -\alpha} \right)\,.
\end{multline*}

Here in the last equation we use the inequality $\alpha < 2 - \frac{\gamma}{2} $.

It remains to note that for large enough $n$ we have
$$
\frac{3\,\gamma}{2(2-\gamma)} \left(\frac{3 \E\xi}{4}\right)^{1-\gamma/2} n^{2-\gamma/2} L\left(\sqrt n \right) \le n^{2-\gamma/2+\delta}\,.
$$

This concludes the proof of Lemma~\ref{lem:small}

\end{proof}

Theorem~\ref{thm:cluster} follows immediately from Lemmas~\ref{lem:p2}, \ref{lem:large}, \ref{lem:small}, and Equation~\eqref{eq:clust}.

\section{Experiments}\label{sec:experiments}

In the previous section, we proved that for any sequence of graphs with a regularly varying degree distribution with a parameter $1 < \gamma < 2$ the global clustering coefficient tends to zero at least as fast as
$n^{-\frac{\left(\gamma - 2\right)^2}{2\gamma}}$.
In this case the number of pairs of adjacent edges is superlinear in the number of vertices and it grows faster than the number of triangles.


In this section, we present a simple method which allows to construct scale-free graphs with a superlinear number of triangles.
Consider a sequence of graphs constructed according to Havel-Hakimi algorithm.
On Figure~\ref{fig:triangles} we present the number of triangles, the number of pairs of adjacent edges, and the global clustering coefficient for such graphs.
For each $n$ we averaged the results over 100 independent samples of power-law degree distribution.
Note that for $\gamma > 2$ the number of pairs of adjacent edges grows linearly and  for $1<\gamma<2$ it grows as $n^{2/\gamma}$, as expected.
The number of triangles grows linearly for $\gamma > 2$ and grows as $n^{3/(\gamma+1)}$ for $1<\gamma<2$.
The constant $3/(\gamma+1)$ can be explained in the following way.
If the degree distribution follows the power law with a parameter $\gamma$, then the maximum clique which can be obtained is of size $n^{\frac{1}{\gamma+1}}$ since $d_k \approx k$ for $k \sim n^{\frac{1}{\gamma+1}}$. This clique gives ${k \choose 3}$ triangles.
Since Havel-Hakimi algorithm also connects the vertices of largest degrees to each other, we get $\sim n^{3/(\gamma+1)}$ triangles.

To sum up, we can construct a sequence of graphs with $n^{\frac{3}{\gamma+1}}$ triangles and our theoretical upper bound is $n^{2 - \frac \gamma 2}$. It is easy to see that for $1 < \gamma < 2$ we have $\frac{3}{\gamma+1} < 2 - \frac \gamma 2$. So, there is a gap between the number of constructed triangles and the upper bound.

\begin{figure}
\includegraphics[width = 0.8\textwidth]{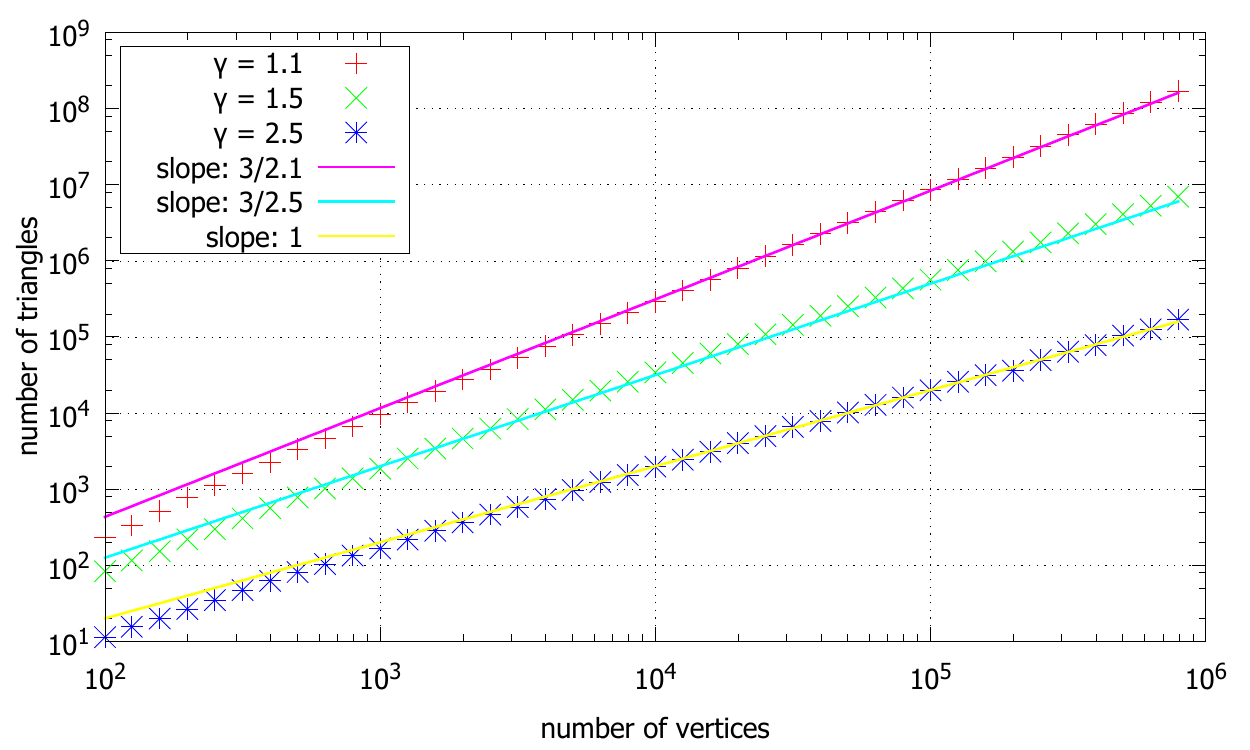}
\includegraphics[width = 0.8\textwidth]{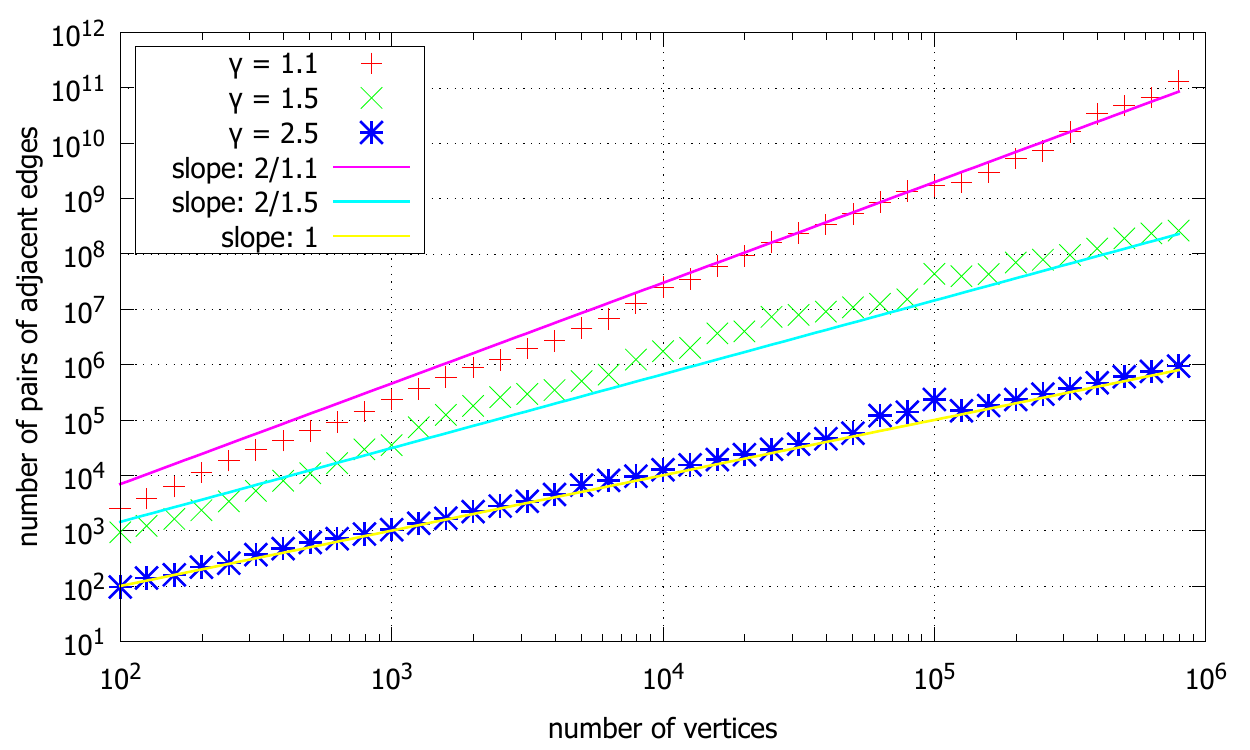}
\includegraphics[width = 0.8\textwidth]{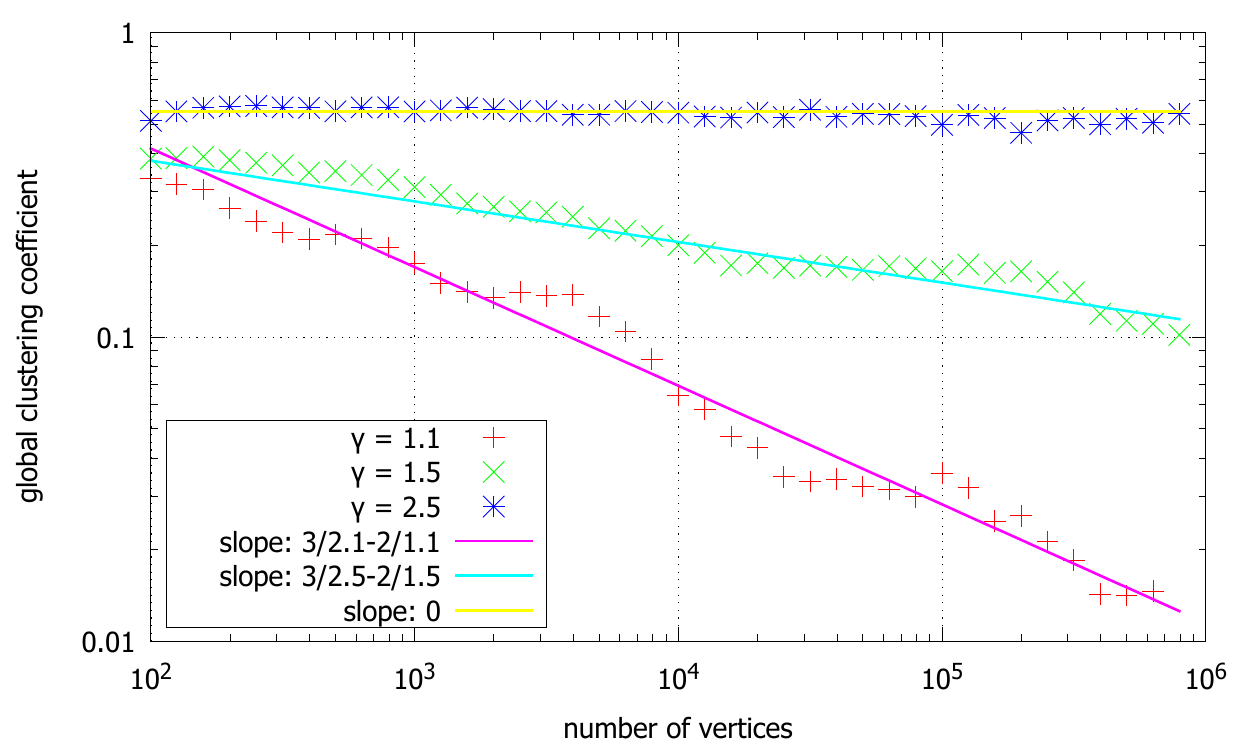}
\caption{Global clustering coefficient for graphs constructed according to Havel-Hakimi algorithm}
\label{fig:triangles}
\end{figure}

\section{Conclusion}\label{sec:conclusion}

In this paper, we analyzed the global clustering coefficient in scale-free graphs.
We proved that for any sequence of graphs with a regularly varying degree distribution with a parameter $1 < \gamma < 2$ the global clustering coefficient tends to zero with high probability.
We also proved that with high probability a graph with the required degree distribution exists.

Finally, we demonstrated the construction procedure which allows to obtain the sequence of graphs with superlinear number of triangles. Unfortunately, the number of triangles in this case grows slower than the upper bound obtained in Section~\ref{sec:global_clustering}.

\end{document}